\documentclass[12pt, reqno]{amsart}
\usepackage{amsmath, amsthm, amscd, amsfonts, amssymb, graphicx}
\usepackage[bookmarksnumbered, colorlinks, plainpages]{hyperref}
\input{mathrsfs.sty}

\def\!{\,!\,}

\newtheorem{theorem}{Theorem}[section]

\newtheorem{corollary}[theorem]{Corollary}
\theoremstyle{definition}

\theoremstyle{remark}
\newtheorem{remark}[theorem]{Remark}

\numberwithin{equation}{section}

\begin{document}

\title{Operator inequalities related to weak $2$-positivity}
\author[M.S. Moslehian and J.I. Fujii]{Mohammad Sal Moslehian$^1$ and Jun Ichi Fujii$^2$}

\address{$^1$ Department of Pure Mathematics, Center of Excellence in
Analysis on Algebraic Structures (CEAAS), Ferdowsi University of
Mashhad, P.O. Box 1159, Mashhad 91775, Iran.}
\email{moslehian@ferdowsi.um.ac.ir and moslehian@member.ams.org}
\address{$^2$ Department of Art and Sciences (Information Science), Osaka Kyoiku University, Asahigaoka, Kashiwara, Osaka 58$2$-8582, Japan.}
\email{fujii@cc.osaka-kyoiku.ac.jp}
\keywords{Operator inequality, weakly $2$-positive map, operator geometric mean, Hua's inequality, Cauchy--Schwarz inequality.}

\subjclass[2010]{Primary 47A63; Secondary 46L05, 47A30.}

\begin{abstract}
In this paper we introduce the notion of weak $2$-positivity and
present some examples. We establish some operator Cauchy--Schwarz
inequalities involving the geometric mean and give some
applications. In particular, we present some operator versions of
Hua's inequality by using the Choi--Davis--Jensen inequality.

 \end{abstract} \maketitle

\section{Introduction}
Let $\mathbb{B}({\mathscr H}), \langle \cdot,\cdot\rangle)$ stand
for the algebra of all bounded linear operators on a complex Hilbert
space ${\mathscr H}$ and let $I$ denote the identity operator. In
the case when $\dim \mathscr{H} =n$, we identify
$\mathbb{B}(\mathscr{H})$ with the full matrix algebra
$\mathcal{M}_n$ of all $n \times n$ matrices with entries in the
complex field $\mathbb{C}$. An operator $ A\in
\mathbb{B}(\mathscr{H})$ is called positive (positive-semidefinite
for matrices) if $\langle A\xi, \xi\rangle \geq 0$ holds for every
$\xi\in \mathscr{H}$ and then we write $A\geq 0$. For self-adjoint
operators $A,B \in \mathbb{B}(\mathscr{H})$, we say $A\leq B$ if
$B-A\geq0$. A map $\Phi: \mathbb{B}({\mathscr H}) \to
\mathbb{B}({\mathscr K})$ is said to be {\it positive} if $\Phi(A)
\geq 0$ whenever $A \geq 0$. A map $\Phi:\mathbb{B}({\mathscr H})
\to \mathbb{B}({\mathscr K})$ is called {\it $2$-positive} if the
map $\Phi_2:\mathbb{B}({\mathscr H}\oplus{\mathscr H}) \to
\mathbb{B}({\mathscr K}\oplus{\mathscr K})$ defined by
$\Phi_2([A_{ij}]_{2\times 2})=[\Phi(A_{ij})]_{2\times 2}$ takes each
positive block matrix to a positive one. If $\Phi_2$ preserves the
positivity of each block matrix of the form
$\begin{bmatrix}A&C\\C&B\end{bmatrix}$, then we call $\Phi$ {\it
weakly $2$-positive}. We say that $\Phi: \mathbb{B}({\mathscr H})
\to \mathbb{B}({\mathscr K})$ is a $*$-map if $\Phi(A^*)=(\Phi(A))^*$.
Choi \cite[Corollary 4.4]{CHO} showed that a positive linear map is
weakly $2$-positive. On the other hand, the Moore--Penrose inverse
$^\dagger$ on the matrix algebra ${\mathcal M}_n$ gives a map
$\Phi^\dagger$ defined by $\Phi^\dagger(A)=A^\dagger$, which is a
nonlinear positive map while it is not weakly $2$-positive (and so
not $2$-positive). In fact, since $\Phi^\dagger$ assigns the
inverses to invertible matrices, we have
$$\begin{bmatrix}2I&I\\I&2I\end{bmatrix}=\begin{bmatrix}2&1\\1&2\end{bmatrix}\otimes I\ge0\quad\text{while}\quad
\begin{bmatrix}\Phi^\dagger(2I)&\Phi^\dagger(I)\\\Phi^\dagger(I)&\Phi^\dagger(2I)\end{bmatrix}=\begin{bmatrix}\frac{1}{2}&1\\1&\frac{1}{2}\end{bmatrix}\otimes I\not\ge0.$$

Next we present a non-trivial example of a weakly $2$-positive map,
which is not $2$-positive. Let us recall a useful criterion due to
Ando \cite[Theorem I.1]{AND}. It states that a block matrix
$T=\begin{pmatrix}A&C\\C^*&B\end{pmatrix}$ is positive if and only
if there exists a contraction $W$ such that
$C=A^{\frac{1}{2}}WB^{\frac{1}{2}}$. We first note that the
nonlinear map $X\mapsto(\det X)I$ on $\mathcal{M}_n$ is
$2$-positive. In fact, the condition
$\begin{pmatrix}A&C\\C^*&B\end{pmatrix}\geq0$ implies that
$C=A^{\frac{1}{2}}WB^{\frac{1}{2}}$ for some contraction $W$. Then
$|\det W|\leq 1$ and $\det C=\sqrt{\det A}\det W\sqrt{\det B}$.
Using again the above criterion we conclude that $\Phi$ is
$2$-positive. The map $\Phi_\alpha(X)=X^*+\alpha(\det X)I$ for
$\alpha\geqq0$ is neither linear nor conjugate linear. It is clearly
weakly $2$-positive. Moreover, let
$$A=\begin{pmatrix}1&0\\0&0\end{pmatrix},\ B=\begin{pmatrix}2&2\\2&2\end{pmatrix},\ C=\begin{pmatrix}1&1\\0&0\end{pmatrix}.$$
Then $A^{\frac{1}{2}}=A$, $B^{\frac{1}{2}}=B/2$ and $C=A^{\frac{1}{2}}IB^{\frac{1}{2}}$, so that $\begin{pmatrix}A&C\\C^*&B\end{pmatrix}\ge0$. Noting to $\det A=\det B=\det C=0$, we have
$$\begin{pmatrix}\Phi_\alpha(A)&\Phi_\alpha(C)\\\Phi_\alpha(C^*)&\Phi_\alpha(B)\end{pmatrix}=\begin{pmatrix}A&C^*\\C&B\end{pmatrix}+\alpha \begin{pmatrix}0&0\\0&0\end{pmatrix}=\begin{pmatrix}1&0&1&0\\0&0&1&0\\1&1&2&2\\0&0&2&2\end{pmatrix}$$
which is not positive since its determinant is negative. Therefore
$\Phi_\alpha$ is not 2-positive for any $\alpha\geqq0$. Furthermore,
these matrices $A, B, C$ can be used to show that the transpose map
$\Phi(A)=A^{tr}$ on $\mathcal{M}_2$ is a weakly $2$-positive linear
map that is not $2$-positive.

The geometric mean $A\# B$ of two positive operators $A, B \in
\mathbb{B}({\mathscr H})$ is characterized by Ando \cite{AND}
$$A\# B=\max\left\{X=X^*\in \mathbb{B}({\mathscr H}): ~ \left[\begin{array}{cc}A&X\\X&B\end{array}\right]\geq 0\right\}.$$
Then we immediately have $\Phi(A\# B)\le \Phi(A)\#\Phi(B)$ for any
weakly $2$-positive map $\Phi$. Ando \cite{AND} also characterized
the harmonic mean $A\!B$ by
$$A\! B=\max\left\{X=X^*\in \mathbb{B}({\mathscr H}): ~ \left[\begin{array}{cc}2A&0\\0&2B\end{array}\right]\geq \left[\begin{array}{cc}X&X\\X&X\end{array}\right]\right\}.$$
Then, for a weakly $2$-positive map $\Phi$, we have
$$\left[\begin{array}{cc}\Phi(2A-A\! B)&\Phi(-A\! B)\\\Phi(-A\! B)&\Phi(2B-A\! B)\end{array}\right]\ge0.$$
If $\Phi$ is linear in this case, we have
$$
\left[\begin{array}{cc}2\Phi(A)&0\\0&2\Phi(B)\end{array}\right]\ge\left[\begin{array}{cc}\Phi(A\! B)&\Phi(A\! B)\\\Phi(A\! B)&\Phi(A\! B)\end{array}\right],
$$
and hence $\Phi(A\! B)\le \Phi(A)\!\Phi(B)$ holds, which is shown in  \cite[Cor.IV.1.3]{AND}.

In this note we present operator Cauchy--Schwarz inequalities for
$2$-weakly positive and $2$-positive maps involving the operator
geometric mean and give two operator Hua types inequalities as
application.


\section{Cauchy--Schwarz type inequalities}

One of the fundamental inequalities in mathematics is the
Cauchy--Schwarz inequality. It states that in an inner product space
$(\mathscr{X},\langle\cdot,\cdot\rangle)$
\begin{equation*}
|\langle x,y\rangle| \leq \|x\|\|y\|\qquad(x,y\in \mathscr{X})\,.
\end{equation*}
There are many generalizations and applications of this inequality
for integrals and isotone functionals; see the monograph \cite{DRA}.
Moreover, some Cauchy--Schwarz inequalities for Hilbert space
operators and matrices involving unitarily invariant norms were
given by Joci\'c \cite{JOC1} and Kittaneh \cite{KIT}. Also Joi\c{t}a
\cite{JOI}, Ili\v sevi\'c and Varo\v{s}anec \cite{I-V}, the first
author and Persson  \cite{MOS2}, Arambasi\'c, Baki\'c and the first
author \cite{ABM} have investigated the Cauchy--Schwarz inequality
and its various reverses in the framework of $C^*$-algebras and
Hilbert $C^*$-modules. Tanahashi, A. Uchiyama and M. Uchiyama
\cite{TU} investigated some Schwarz type inequalities and their
converses in connection with semi-operator monotone functions. A
refinement of the Cauchy--Schwarz inequality involving connections
is investigated by Wada \cite{WAD}. An application of the
covariance-variance inequality to the Cauchy--Schwarz inequality was
obtained by Fujii, Izumino, Nakamoto and Seo \cite{nak}. Some
operator versions of the Cauchy--Schwarz inequality with simple
conditions for the case of equality are presented by the second
author \cite{FUJ2}.

To achieve our main result we need the polar decomposition of
bounded linear operators. Recall that if $A\in \mathbb{B}({\mathscr
H})$, then there exists a unique partial isometry $U \in
\mathbb{B}({\mathscr H})$ such that $A=U|A|$ and $\ker(U)=\ker(|A|)$
(the polar decomposition). Then $U^*A=|A|$ and $A^*=U^*|A^*|$ is the
polar decomposition of $A^*$.

\begin{theorem}\label{main}
Let $A, B, X, Y\in \mathbb{B}({\mathscr H})$  be arbitrary
operators.
\begin{itemize}
\item [(i)] If $\Phi$ is a weakly $2$-positive map, then $\Phi(|X^*A^*Y|) \leq \Phi(V^*X^*|A|XV) \,\#\, \Phi(Y^*|A^*|Y)\,,$ in which $X^*A^*Y=V|X^*A^*Y|$ denotes the polar decomposition.
\item[(ii)] If $\Phi$ is a $2$-positive $*$-map, then
$|\Phi(X^*A^*Y)| \leq U^*\Phi(X^*|A|X)U \,\#\, \Phi(Y^*|A^*|Y)\,,$
in which $\Phi(X^*A^*Y)=U|\Phi(X^*A^*Y)|$ denotes the polar
decomposition.
\end{itemize}
\end{theorem}
\begin{proof}
(i) First note that
$$\left[\begin{array}{cc}|A|&A^*\\A&|A^*|\end{array}\right]= \left[\begin{array}{cc}I&0\\0&W\end{array}\right] \left[\begin{array}{cc}|A|^{1/2}&0\\|A|^{1/2}&0\end{array}\right]
\left[\begin{array}{cc} |A|^{1/2}& |A|^{1/2}\\0&0\end{array}\right]
\left[\begin{array}{cc}I&0\\0&W \end{array}\right]^* \geq 0,$$ where
we apply the polar decomposition $A=W|A|$. Hence
\begin{eqnarray}\label{MMM}
\left[\begin{array}{cc}X^*|A|X&X^*A^*Y\\Y^*AX&Y^*|A^*|Y \end{array}\right]= \left[\begin{array}{cc}X^*&0\\0&Y^*\end{array}\right] \left[\begin{array}{cc}|A|&A^*\\A&|A^*|\end{array}\right]
\left[\begin{array}{cc}X&0\\0&Y\end{array}\right]\geq 0\,.
\end{eqnarray}
Utilizing the polar decomposition $X^*A^*Y=V|X^*A^*Y|$ we obtain
\begin{eqnarray*}
\left[\begin{array}{cc}V^*(X^*|A|X)V& |X^*A^*Y|\\ |X^*A^*Y|& Y^*|A^*|Y\end{array}\right]&=& \left[\begin{array}{cc}V^*(X^*|A|X)V& V^*(X^*A^*Y) \\ (Y^*AX)V& Y^*|A^*|Y\end{array}\right]\\
&=& \left[\begin{array}{cc} V& 0\\ 0& I\end{array}\right]^* \left[\begin{array}{cc}X^*|A|X& X^*A^*Y\\Y^*AX& Y^*|A^*|Y\end{array}\right] \left[\begin{array}{cc}V& 0\\ 0 & I\end{array}\right]\\
& \geq& 0\,.
\end{eqnarray*}
Due to the weak $2$-positivity of $\Phi$, we get
\begin{eqnarray*}
\left[\begin{array}{cc}\Phi(V^*(X^*|A|X)V)& \Phi(|X^*A^*Y|)\\ \Phi(|X^*A^*Y|)& \Phi(Y^*|A^*|Y)\end{array}\right] \geq 0\,.
\end{eqnarray*}
Thus we obtain
$$\Phi(|X^*A^*Y|) \leq \Phi(V^*X^*|A|XV) \,\#\, \Phi(Y^*|A^*|Y)\,.$$
(ii) It follows from \eqref{MMM} and $2$-positivity of $\Phi$ that
\begin{eqnarray*}
\left[\begin{array}{cc}\Phi(X^*|A|X)& \Phi(X^*A^*Y)\\ \Phi(Y^*AX)& \Phi(Y^*|A^*|Y)\end{array}\right] \geq 0\,,
\end{eqnarray*}
whence, by using the polar decoposition
$\Phi(X^*A^*Y)=U|\Phi(X^*A^*Y)|$, we get
\begin{align*}
\left[\begin{array}{cc}U^*\Phi(X^*|A|X)U& |\Phi(X^*A^*Y)|\\
|\Phi(X^*A^*Y)|& \Phi(Y^*|A^*|Y)\end{array}\right]&=\left[\begin{array}{cc}U^*\Phi(X^*|A|X)U& U^*\Phi(X^*A^*Y)\\
\Phi(Y^*AX)U& \Phi(Y^*|A^*|Y)\end{array}\right]\\
&=\left[\begin{array}{cc} U& 0\\ 0& I\end{array}\right]^*
\left[\begin{array}{cc}\Phi(X^*|A|X)& \Phi(X^*A^*Y)\\ \Phi(Y^*AX)&
\Phi(Y^*|A^*|Y)\end{array}\right] \left[\begin{array}{cc} U& 0\\ 0&
I\end{array}\right]\\& \geq 0\,,
\end{align*}
which gives the desired inequality.
\end{proof}
\begin{remark}
The proof of Theorem \ref{main}(ii) shows that if $A=A^*$ and $Y=X$,
then the ``$2$-positivity'' of $\Phi$ can be replaced by the weaker
assumption ``weak $2$-positivity''. Then we get (ii)$^\prime$ If
$\Phi$ is a weakly $2$-positive $*$-map, then $|\Phi(X^*AX)| \leq
U^*\Phi(X^*|A|X)U \,\#\, \Phi(X^*|A|X)\,,$ in which
$\Phi(X^*AX)=U|\Phi(X^*AX)|$ denotes the polar decomposition.
\end{remark}
Now consider the separable Hilbert space $H=\ell_2$. Take the
$2$-positive map $\Phi(A)=\langle Ae,e\rangle$ where $A\in
\mathbb{B}({\mathscr H})$, $e=(1, 0, 0, \cdots)$ and
$X=x\otimes\overline{e}$, $Y=y\otimes\overline{e}$ where
$(x\otimes\overline{y})(z):=\langle z,y\rangle x$. Then we get from
Theorem \ref{main} (ii) the following Cauchy--Schwarz inequality in
Hilbert spaces:
\begin{corollary}
Let $A \in \mathbb{B}({\mathscr H})$ and $x,y \in \mathscr{H}$. Then
$$|\langle Ax,y\rangle|^2 \leq \langle |A|x,x\rangle \langle |A^*|y,y\rangle\,.$$
\end{corollary}
Considering the positive linear functional ${\rm tr}(\cdot)$ on
$\mathcal{M}_n$, it follows from Theorem \ref{main}(i) that
\begin{corollary}
Let $A, X, Y \in \mathcal{M}_n$. Then
$${\rm tr}(|X^*A^*Y|)^2 \leq {\rm tr}(X^*|A|X) {\rm tr}(Y^*|A^*|Y)\,.$$
\end{corollary}
\begin{corollary}
Let $X \in \mathbb{B}({\mathscr H})$.
\begin{itemize}
\item[(i)]  If $\Phi$ is a weakly $2$-positive map, then
$\Phi(|X|)\leq \Phi(V^*|X^*|V)\,\#\,\Phi(|X|)\,,$\\ where $X=V|X|$
is the polar decomposition.
\item [(ii)]  If $\Phi$ is a $2$-positive $*$-map, then
$|\Phi(X)| \leq U^*\Phi(|X^*|^{1/2})U\,\#\,\Phi(|X|^{3/2})\,,$\\
where $\Phi(X)=U|\Phi(X)|$ is the polar decomposition.
\end{itemize}
\end{corollary}
\begin{proof}
Let $X=V|X|$ be the polar decomposition of $X$. It follows from
Theorem \ref{main} (i) that
\begin{itemize}
\item[(i)]
$\Phi(|X|)= \Phi(|V^*XI|) \leq
\Phi(IV^*|X^*|VI)\,\#\,\Phi(|X|)=\Phi(V^*|X^*|V)\,\#\,\Phi(|X|)\,.$
\item [(ii)] Utilizing Theorem \ref{main} (ii) we have
\begin{eqnarray*}
|\Phi(X)| &=& |\Phi(V|X|^{1/2}|X|^{1/2})|\\
&\leq& U^*\Phi(V|X|^{1/2}V^*)U\,\#\, \Phi(|X|^{1/2}|X|^{1/2}|X|^{1/2})\\
&=& U^*\Phi(|X^*|^{1/2})U\,\#\,\Phi(|X|^{3/2})\,.
\end{eqnarray*}
\vspace*{-40pt}\end{itemize}
\end{proof}

\section{Applications to Hua's inequality}

Hua's inequality states that
$$\left(\delta- \sum^n_{i=1} x_i\right)^2+ \alpha \sum^n_{i= 1}
x^2_i\geq \frac{\alpha}{n+ \alpha}\delta^2\,,$$ where $\delta$,
$\alpha$ are positive numbers and $x_i$ $(i= 1,2,\dots, n)$ are real
numbers. There are several refinement and improvement of this
inequality in the literature; see \cite{MOS} and references therein.
An operator version of Hua's inequality was given by Drnov\v sek
\cite{DRN}. Moreover, Radas and \v{S}iki\'c \cite{R-S} generalized
the Hua inequality for linear operators in real inner product
spaces. A refinement of Hua's inequality was presented by the second
author in \cite{FUJ} by showing that if $A, B$ are bounded linear
operators acting on a Hilbert space $\mathscr{H}$ and $\varphi$ is a
state on $\mathbb{B}(\mathscr{H})$, then
\begin{eqnarray}\label{fujhua}
(1-|\varphi(B^*A)|)^2 \geq (1-\sqrt{\varphi(A^*A)\varphi(B^*B)})^2
\geq \varphi(I-A^*A)\varphi(I-B^*B),
\end{eqnarray}
which in turn gives an extension of the above classical Hua's
inequality by considering $\varphi$ as the normalized trace on the
matrix algebra $\mathcal{M}_n$ and some suitable diagonal matrices.
An extension in the setting of Hilbert $C^*$-modules and operators
on Hilbert spaces was given by the first author in \cite{MOS}.

Our first result in this section gives an extension of
\eqref{fujhua}. Recall that a contraction is an operator $A$ of norm
less than or equal one.

\begin{theorem}
Let $\Phi$ be a $2$-positive $*$-map and let $A, B, X, Y\in
\mathbb{B}({\mathscr H})$  be arbitrary operators. If
$\Phi(X^*A^*Y)=U|\Phi(X^*A^*Y)|$ is the polar decomposition of
$\Phi(X^*A^*Y)$, and $\Phi(Y^*|A^*|Y)$ and $\Phi(X^*|A|X)$ are
contractions, then
$$I-|\Phi(X^*A^*Y)| \geq U^*\big(I-\Phi(X^*|A|X)\big)U\, \#\, \big(I-\Phi(Y^*|A^*|Y)\big)\,.$$
\end{theorem}
\begin{proof}
Theorem \ref{main} (ii) ensures that
\begin{eqnarray}\label{oct}
I-|\Phi(X^*A^*Y)|\geq I- \Big(U^*\Phi(X^*|A|X)U\, \#\,
\Phi(Y^*|A^*|Y)\Big)\,.
\end{eqnarray}
Using the properties of the geometric mean (see \cite[Chapter
5]{seo}), we get
\begin{eqnarray*}
&&\hspace{-1.5cm} \Big(U^*(I-\Phi(X^*|A|X))U \,\#\, \big(I-\Phi(Y^*|A^*|Y)\big)\Big) + \Big(U^*\Phi(X^*|A|X)U \,\#\, \Phi(Y^*|A^*|Y)\Big) \\
&\leq& U^*U \# I \,\,\qquad\qquad\qquad\quad\qquad\qquad\quad
(\mbox{by the subadditivity of the
geometric mean})\\
&\leq& I \# I\qquad\qquad\qquad\qquad\qquad\qquad\qquad(\mbox{by the
monotonicity of the
geometric mean})\\
&=& I\,,
\end{eqnarray*}
which together with \eqref{oct} give the required inequality.
\end{proof}

Now let $f$ be a continuous real function $f$ defined on an interval
$\mathcal{J}\subseteq\mathbb{R}$. The function $f$ is called {\it
operator convex} if
$$f\left(\frac{A+B}{2}\right)\le \frac{f(A)+f(B)}{2}$$
for all selfadjoint operators $A$ and $B$ with spectra contained in
$\mathcal{J}$. There are several statements equivalent to the
operator convexity; see \cite[Theorems 1.9 and 1.10]{seo}. In
particular, $f$ is operator convex if and only if
\begin{eqnarray}\label{jen1}
f\left(\sum_{i=1}^nX_i^*A_iX_i\right) \leq
\sum_{i=1}^nX_i^*f(A_i)X_i
\end{eqnarray} for all self-adjoint bounded operators $A_i$ with
spectra contained in $\mathcal{J}$ and all bounded operators $X_i$
with $\sum_{i=1}^nX_i^*X_i=I$; cf \cite{H-P}. The Jensen operator
inequality due to Davis \cite{DAV} and Choi \cite{CH} reads as
follows
$$f(\Phi(A))\le\Phi(f(A))\qquad\qquad (\mbox{The Choi--Davis--Jensen inequality})$$
where $\Phi$ is a unital positive linear map on
$\mathbb{B}(\mathscr{H})$, $f$ is operator convex and $A$ is a
self-adjoint operator whose spectrum ${\rm sp}(A)$ is contained in
$\mathcal{J}$.

Finally we show another type of Hua's operator inequality. Recall
that a {\it conditional expectation} $\Phi$ from a unital
$C^*$-algebra $\mathscr{A}$ of operators to a $C^*$-subalgebra
$\mathscr{B}$ of $\mathscr{A}$ containing its identity is a linear
norm reducing idempotent. Such a map is  completely positive and
satisfies the bimodule property $\Phi(AXB)=A\Phi(X)B$ for all
$A,B\in\mathscr{B}$ and $X\in \mathscr{A}$.

\begin{theorem} Let $f$ be an operator convex function on an interval $\mathcal{J}$ and $\Phi$ be a conditional expectation from a unital $C^*$-algebra
$\mathscr{A}$ of operators to a $C^*$-subalgebra $\mathscr{B}$ of
$\mathscr{A}$ containing its identity. If $C\in\mathscr{B}$ is
invertible and $B\in\mathscr{A}$ is self-adjoint and satisfies
$${\rm sp}(I-\Phi(B))\cup{\rm sp}((I+C^*C)^{-1})\cup{\rm sp}({C^*}^{-1}BC^{-1})\subseteq\mathcal{J},$$
then
$$f(I-\Phi(B))+C^*\Phi\left(f\big({C^*}^{-1}BC^{-1}\big)\right)C\geq f\left((I+C^*C)^{-1}\right)(I+C^*C)\,.
$$
\end{theorem}

\begin{proof}
Put
$$X=(I+C^*C)^{-1/2} \quad{and}\quad Y=C(I+C^*C)^{-1/2}.$$
Then $X^*X+Y^*Y=I$. We have
\begin{align*}
f(I-\Phi(B))&+C^*\Phi(f({C^*}^{-1}BC^{-1}))C\\
&=X^{-1}\left[Xf(I-\Phi(B))X+Y^*\Phi(f({C^*}^{-1}BC^{-1}))Y\right]X^{-1}\\
&\geq X^{-1}\left[Xf(I-\Phi(B))X+Y^*f(\Phi({C^*}^{-1}BC^{-1}))Y\right]X^{-1}\\
&\qquad\qquad\qquad\qquad\qquad\qquad\quad (\mbox{ by the Choi--Davis--Jensen inequality})\\
&= X^{-1}\left[Xf(I-\Phi(B))X+Y^*f({C^*}^{-1}\Phi(B)C^{-1})Y\right]X^{-1}\\
&\qquad\qquad\qquad\qquad\qquad\qquad\qquad\quad\quad (\mbox{by the bimodule property of~} \Phi)\\
&\geq X^{-1}f\Big(X(I-\Phi(B))X+Y^*{C^*}^{-1}\Phi(B)C^{-1}Y\Big)X^{-1}\\
&\qquad\qquad\qquad\qquad\qquad\qquad\qquad\qquad\quad (\mbox{by~} X^*X+Y^*Y=I \mbox{ and } \eqref{jen1})\\
&= X^{-1}f\Big(X(I-\Phi(B))X+X\Phi(B)X\Big)X^{-1}\\
&= f(X^2)X^{-2}\qquad\qquad\qquad\qquad\qquad\quad\quad (\mbox{by
the functional calculus})\\
&= f\left((I+C^*C)^{-1}\right)(I+C^*C)\,.
\end{align*}
\end{proof}

\begin{corollary} Let $f$ be an operator convex function on an interval $\mathcal{J}$, $\varphi$ be a state and $\gamma>0$.
If $B$ is self-adjoint, $1-\varphi(B)$ and $1/(\gamma+1)$ belong to
$\mathcal{J}$ and ${\rm sp}(B/\gamma)\subseteq \mathcal{J}$, then
$$f(1-\varphi(B))+\gamma\varphi\big(f(B/\gamma)\big)\geq(1+\gamma)f\left({1\over1+\gamma}\right).$$
\end{corollary}

\end{document}